\documentclass[12pt,reqno]{amsart}
\usepackage{amssymb}
\usepackage{amsfonts}
\usepackage{hyperref}

\def\equationautorefname~#1\null{(#1)\null}

\theoremstyle{plain}

\newtheorem{algorithm}{Algorithm}[section]
\newtheorem{thm}{Thm}

\newtheorem*{Remark-delta}{Remark on all things $\protect\delta$}

\newtheorem{lemma}[algorithm]{Lemma}

\newtheorem{theoremlet}[thm]{Theorem}
\newtheorem{corollarylet}[thm]{Corollary}
\newtheorem{lemmalet}[thm]{Lemma}

\newtheorem*{remarknonum}{Remark}

\newtheorem*{definitionnonum}{Definition}

\newtheorem{proposition}[algorithm]{Proposition}

\numberwithin{equation}{algorithm}

\input{tcilatex}

\begin{document}
\title{Random Manifolds have no Totally Geodesic Submanifolds}
\author{Thomas Murphy}
\address{Tommy Murphy, Dept. of Mathematics, California State University
Fullerton\\
Fullerton CA 92831}
\email[tmurphy]{tmurphy@fullerton.edu}
\urladdr{http://mathfaculty.fullerton.edu/tmurphy}
\author{Frederick Wilhelm}
\thanks{This work was supported by a grant from the Simons Foundation
(\#358068, Frederick Wilhelm)}
\address{Fred Wilhelm, Dept. of Mathematics, University of California
Riverside, Riverside, Ca 92521. }
\email{fred@math.ucr.edu}
\urladdr{https://sites.google.com/site/frederickhwilhelmjr/home}
\date{March 2017.}

\begin{abstract}
For $n\geq 4$ we show that generic closed Riemannian $n$--manifolds have no
nontrivial totally geodesic submanifolds, answering a question of Spivak. An
immediate consequence is a severe restriction on the isometry group of a
generic Riemannian metric. Both results are widely believed to be true, but
we are not aware of any proofs in the literature.
\end{abstract}

\maketitle

Schoen-Simon showed that every Riemannian manifold admits an embedded
minimal hypersurface (\cite{SchSim}, cf. also \cite{Pitts}). Intuition
suggests that the analogous result about totally geodesic submanifolds is
false. In fact, Spivak writes that it

\begin{quotation}
\textquotedblleft \textit{seems rather clear that if one takes a Riemannian
manifold }$\left( N,\left\langle \cdot ,\cdot \right\rangle \right) $\textit{%
\ `at random', then it will not have any totally geodesic submanifolds of
dimension $>1$. But I must admit that I don't know of any specific example
of such a manifold.\textquotedblright } (\cite{spivak}, p. 39)
\end{quotation}

The existence of specific examples was established by Tsukada in \cite%
{tsukada}, who found some left-invariant metrics on $3$--dimensional Lie
groups without totally geodesic surfaces. In the present paper, we prove
that Spivak's intuition about generic metrics is correct for compact
Riemannian $n$--manifolds with $n\geq 4$.

\begin{theoremlet}
\label{main thm}Let $M$ be a compact, smooth manifold of dimension $\geq 4$.
For any finite $q\geq 2,$ the set of Riemannian metrics on $M$ with no
nontrivial immersed totally geodesic submanifolds contains a set that is
open and dense in the $C^{q}$--topology.
\end{theoremlet}

Put another way: in a generic Riemannian $n$--manifold with $n\neq 3,$ any
totally geodesic submanifold is either a geodesic or the whole manifold. We
emphasize that this statement applies to all immersed submanifolds---there
is no requirement that the submanifolds be closed or complete.

In \cite{Eb}, Ebin showed most Riemannian manifolds have no isometries other
than the identity. Theorem \ref{main thm} yields a simple, alternative proof
of this for most group actions.

\begin{corollarylet}
Let $M$ be a compact, smooth manifold of dimension $\geq 4.$ Let $G$ act
smoothly and effectively on $M$ so that either:

\begin{enumerate}
\item A subgroup of $G$ has a fixed point set of dimension $\geq 2,$ or

\item $G$ has a subgroup $H$ whose fixed point set is $0$ or $1$
dimensional, and $H$ does not act freely and linearly on a sphere.
\end{enumerate}

Then for any finite $q\geq 2,$ the set of Riemannian metrics on $M$ that are
not $G$--invariant contains a set that is open and dense in the $C^{q}$%
--topology.
\end{corollarylet}

To see how this follows from Theorem \ref{main thm}, suppose that $G$ acts
isometrically and effectively on a Riemannian manifold that has no
nontrivial immersed totally geodesic submanifolds. Then the fixed point sets
of $G$ and all of its subgroups have dimension $\leq 1.$ If a subgroup $H$
has a one dimensional fixed point set, then since no subgroup of $H$ can
have a larger fixed point set, $H$ acts freely on any unit normal sphere to
its fixed point set. If no subgroup of $G$ has a one dimensional fixed point
set, but some subgroup $H$ has a zero dimensional fixed point set, then
differentiating $H$ produces a free action on the unit tangent sphere at any
fixed point of $H.$ In particular, if $G$ has a subgroup, $H,$ with a one
dimensional fixed point set, then $H$ is either discrete or isomorphic to $%
S^{1},$ $S^{3},$ or to a $\mathbb{Z}_{2}$--extension of $S^{1},$ and if $H$
is discrete, then it is the fundamental group of a complete manifold of
constant curvature $1.$ $H$ also satisfies these constraints if it has a $0$%
--dimensional fixed point set and does not contain a subgroup with a $1$%
--dimensional fixed point set.

It seems rather easy to construct a deformation that kills the totally
geodesic property for a fixed submanifold or a fixed compact family of
submanifolds (see, e.g., \cite{bryant}). Although there are compactness
theorems for submanifolds with constrained geometry in, e.g., \cite{GuijW},
the space of all submanifolds of a compact Riemannian manifold is not
compact. For example, via the Nash isometric embedding theorem, all
Riemannian manifolds of any fixed dimension $k$ embed isometrically into a
fixed flat $n$--torus if $n>>k.$

To circumvent this difficulty we propose a new concept called \emph{%
partially geodesic}.\emph{\ }It is defined in terms of the following
invariant of self adjoint linear maps.

\begin{definitionnonum}
Let $\Phi :V\longrightarrow V$ be a self adjoint linear map of an inner
product space $V$. For a subspace $W$ of $V,$ we set 
\begin{equation}
\mathcal{I}_{\Phi }\left( W\right) \equiv \max_{\left\{ \left. w\in W\text{ }%
\right\vert \text{ }\left\vert w\right\vert =1\right\} }\left\vert \Phi
\left( w\right) ^{W^{\perp }}\right\vert ,  \label{I measu}
\end{equation}%
where $\Phi \left( w\right) ^{W^{\perp }}$ is the component of $\Phi \left(
w\right) $ that is perpendicular to $W.$
\end{definitionnonum}

Let $V=T_{p}M$ be a tangent space to a Riemannian manifold $\left(
M,g\right) .$ For $v\in W\subset T_{p}M,$ the Jacobi operator $R_{v}=R(\cdot
,v)v:T_{p}M\longrightarrow T_{p}M$ is self adjoint with respect to $g$, and
if 
\begin{equation*}
\mathcal{I}_{R_{v}}\left( W\right) \neq 0
\end{equation*}%
for some $v\in W,$ then $W$ is not tangent to any totally geodesic
submanifold. This motivates the following concept.

\begin{definitionnonum}
For $l\in \left\{ 2,3,\ldots ,n-1\right\} ,$ an $l$--plane $P$ tangent to a
Riemannian $n$--manifold $M$ is called partially geodesic if and only for
all $v\in P,$ 
\begin{equation*}
\mathcal{I}_{R_{v}}\left( P\right) =0.
\end{equation*}
\end{definitionnonum}

Theorem \ref{main thm} is a consequence of

\begin{theoremlet}
\label{partiall geod thm}Let $M$ be a compact, smooth manifold of dimension $%
\geq 4$. For any finite $q\geq 2,$ the set of Riemannian metrics on $M$ with
no partially geodesic $l$--planes is open and dense in the $C^{q}$--topology.
\end{theoremlet}

Since $q\geq 2$, the curvature tensor is continuous in the $C^{q}$%
--topology. Combined with the compactness of the Grassmannians of $l$%
--planes, it follows that that the set of metrics with no partially geodesic 
$l$--planes is open.

Since the $C^{2}$--topology is finer than the $C^{0}$ and $C^{1}$%
--topologies, it follows from Theorem \ref{partiall geod thm} that the set
of metrics with no partially geodesic $l$--planes is dense in the $C^{0}$
and $C^{1}$--topologies; however, as the curvature tensor is not continuous
in these topologies, it seems likely that the openness assertion fails.

The balance of the paper is therefore devoted to proving the density
assertion of Theorem \ref{partiall geod thm}. To do this we use reverse
induction on $l$ via the following statement.

\bigskip

\noindent $\mathbf{l}^{th}$--\textbf{Partially Geodesic Assertion. }\emph{\
Given }$l\in \left\{ 2,3,\ldots ,n-1\right\} ,$\emph{\ a finite }$q\geq 2,$ 
\emph{and }$\xi >0$, \emph{there is a Riemannian metric }$\tilde{g}$ \emph{on%
} $M$ \emph{that has no partially geodesic }$k$\emph{--planes for all }$k\in
\left\{ l,l+1,\ldots ,n-1\right\} $ \emph{and satisfies }%
\begin{equation*}
\left\vert \tilde{g}-g\right\vert _{C^{q}}<\xi .
\end{equation*}

\bigskip

The rest of the paper is devoted to proving this assertion. To do so, we
exploit a principle given in the following lemma.

\begin{lemmalet}
\label{uber stru lemma}Let $\left\{ g_{s}\right\} _{s\geq 0}$ be a smooth
family of Riemannian metrics on $M.$ Let $R^{s}$ be the curvature tensor of $%
g_{s}.$ Let $\mathcal{P}_{0}$ be the set of partially geodesic $l$--planes
for $g_{0},$ and suppose that for all $k\in \left\{ l+1,\ldots ,n-1\right\}
, $ $g_{0}$ has no partially geodesic $k$--planes.

Suppose further that there are $c,s_{0}>0$ and a neighborhood $\mathcal{U}%
_{0}$ of $\mathcal{P}_{0}$ so that for every $P\in \mathcal{U}_{0}$, every $%
s\in \left( 0,s_{0}\right) ,$ and some $g_{0}$--unit $v\in P,$ 
\begin{equation}
\mathcal{I}_{R_{v}^{s}}\left( P\right) >cs.  \label{unif pos}
\end{equation}

Then for all sufficiently small $s,$ and all $k\in \left\{ l,\ldots
,n-1\right\} $, $\left( M,g_{s}\right) $ has no partially geodesic $k$%
--planes.
\end{lemmalet}

\begin{proof}
We write $\mathcal{G}_{k}\left( M\right) $ for the Grassmannian of $k$%
--planes tangent to $M.$ Since each $\mathcal{G}_{k}\left( M\right) $ is
compact, there is a $\delta >0$ so that for all $k\in \left\{ l+1,\ldots
,n-1\right\} $ and all $P\in \mathcal{G}_{k}\left( M\right) $ there is a
unit $v\in P$ so that 
\begin{equation*}
\mathcal{I}_{R_{v}^{0}}\left( P\right) >\delta .
\end{equation*}

Similarly $\mathcal{G}_{l}\left( M\right) \setminus \mathcal{U}_{0}$ is
compact. Thus there is a (possibly different) $\delta >0$ so that for all $%
P\in \mathcal{G}_{l}\left( M\right) \setminus \mathcal{U}_{0}$ there is a
unit $v\in P$ so that 
\begin{equation*}
\mathcal{I}_{R_{v}^{0}}\left( P\right) >\delta .
\end{equation*}%
By combining the previous two displays with Inequality (\ref{unif pos}) and
a continuity argument, it follows that for all sufficiently small $s,$ $%
\left( M,g_{s}\right) $ has no partially geodesic $l$--planes.
\end{proof}

In Section \ref{not and conve}, we establish notations and conventions. In
Section \ref{local constr}, we prove Lemma \ref{local constr main lemma},
which implies that the $l^{th}$--Partially Geodesic Assertion holds locally,
in a sense that is quantifiable. This allows us, in Section \ref{global
constr}, to piece together various local deformations and complete the proof
of Theorem \ref{partiall geod thm}. We do not use Lemma \ref{uber stru lemma}
explicitly, but the reader will notice that a similar principle is used in
our global argument in Section \ref{global constr}.

For a quick overview of the proof, imagine that $v$ and $T$ are tangent to a
partially geodesic plane $P$ and $n\perp P.$ The strategy is to change $%
\left\langle T,n\right\rangle $ by a function $f$ that has a relatively
large $2^{nd}$ derivative in the $v$--direction. This has the effect of
giving $R\left( T,v\right) v$ a component in the $n$--direction. In
particular, $P$ is no longer partially geodesic. Since this is a local
deformation, $f$ has compact support and necessarily has inflection points.
To deal with this, we simultaneously change two components of the metric
tensor using two functions whose inflection points occur at different
places. Since this construction requires the presence of two distinct
orthonormal triples, it only works in dimensions $\geq 4.$ Our sense is that
a modification of our ideas might also yield a proof of Theorems \ref{main
thm} and \ref{partiall geod thm} in dimension 3. In fact, Bryant has
outlined a local proof in \cite{bryant}.

\begin{remarknonum}
As mentioned above, Theorem \ref{main thm} is widely believed to be true. In 
\cite{berger}, Berger wrote (without proof) \textquotedblleft a generic
Riemannian manifold does not admit any such submanifold\textquotedblright .

Hermann states in \cite{hermann} that Theorem \ref{main thm} should be true
but that there is little research in this direction.
\end{remarknonum}

\begin{remarknonum}
In \cite{SchSim}, Schoen-Simon showed that every Riemannian $n$--manifold
admits an embedded minimal hypersurface. If $n\geq 8,$ the Schoen-Simon
construction can lead to minimal hypersurfaces with singularities. By
contrast, Theorem \ref{main thm} rules out the possibility of a generic
metric having any totally geodesic submanifold, complete or otherwise. In
particular, generic Riemannian manifolds, of dimension $\geq 4,$ have no
totally geodesic submanifolds with singularities.
\end{remarknonum}

\begin{remarknonum}
Theorem \ref{main thm} asserts that the set of metrics with no totally
geodesic submanifolds has nonempty interior. On the other hand, Theorem \ref%
{partiall geod thm} says that the set of metrics with no partially geodesic
submanifolds is an actual open set in the $C^{q}$--topology$.$ It is not
clear to us whether the set of metrics with no totally geodesic submanifolds
is open. As mentioned above, one difficulty is that the space of
isometrically embedded $k$--manifolds in an $n$--manifold is not compact.
\end{remarknonum}

\bigskip

\noindent \textbf{Acknowledgement: }\emph{We are grateful to referees for so
thoroughly reading the manuscript, to Jim Kelliher, Catherine Searle and the
referees for valuable suggestions, and to Paula Bergen for copyediting the
manuscript. }

\section{Notations and Conventions\label{not and conve}}

Throughout, $(M,g)$ will be a smooth, connected compact Riemannian manifold
of dimension $n\geq 4.$ We will denote the Levi-Civita connection, curvature
tensor, and Christoffel symbols by $\nabla $, $R$, and $\Gamma ,$
respectively. We adopt the sign convention that $R_{xyyx}$ is the sectional
curvature of a plane spanned by orthonormal $x,y\in T_{p}M$. Thus the Jacobi
operator is $R_{v}=R(\cdot ,v)v:T_{p}M\longrightarrow T_{p}M.$ For a nearby
metric $\tilde{g}$, the corresponding objects will be denoted $\tilde{\nabla}
$, $\tilde{R}$, and $\tilde{\Gamma},$ respectively.

Given local coordinates $\left\{ x_{i}\right\} _{i=1}^{n}$, define $\partial
_{i}$ to be the partial derivative in the direction $\frac{\partial }{%
\partial x_{i}}$. At times the notation $\partial _{x_{i}}$ will also be
used for the same object. We let $\mathcal{G}_{l}\left( M\right) $ denote
the Grassmannian of $l$-planes in $M$, and $\pi :\mathcal{G}_{l}\left(
M\right) \longrightarrow M$ the projection of $\mathcal{G}_{l}\left(
M\right) $ to $M$. We fix a Riemannian metric on $\mathcal{G}_{l}\left(
M\right) $ so that $\pi :\mathcal{G}_{l}\left( M\right) \longrightarrow
\left( M,g\right) $ is a Riemannian submersion with totally geodesic fibers
that are isometric to the Grassmannian of $l$-planes in $\mathbb{R}^{n}.$
For a metric space $X$, $A\subset $ $X$, and $r>0,$ we let 
\begin{equation*}
B\left( A,r\right) \equiv \left\{ \left. x\in X\text{ }\right\vert \text{ 
\textrm{dist}}\left( x,A\right) <r\right\} .
\end{equation*}

For some $l\in \left\{ 2,\ldots ,n-1\right\} $ we let $\mathcal{P}_{0}$ be
the set of partially geodesic $l$--planes for $g$.

\section{ The Local construction\label{local constr}}

In this section, we prove Lemma \ref{local constr main lemma}, which can be
viewed as a local version of the $l^{th}$--Partially Geodesic Assertion. In
Section \ref{global constr}, we exploit the fact that $\mathcal{P}_{0}$ is
compact and apply Lemma \ref{local constr main lemma} successively to each
element of a finite open cover $\left\{ O_{i}\right\} _{i}^{G}$ of $\mathcal{%
P}_{0}$. This will produce a finite sequence of metrics $g_{1},g_{2},\ldots
,g_{G}$ where, for example, $g_{2}$ is obtained by applying Lemma \ref{local
constr main lemma} to $g_{1}.$ The idea is that Lemma \ref{local constr main
lemma} kills the partially geodesic property on $O_{k}$ while simultaneously
preserving it on $\cup _{i=1}^{k-1}O_{i}.$ In particular, the set of
partially geodesic $l$--planes for $g_{k}$ is contained in $\cup
_{i=k+1}^{G}O_{i}.$

Because of the successive nature of our construction, in Lemma \ref{local
constr main lemma} we construct a deformation, not of $g,$ but rather of an
abstract metric, $\hat{g},$ that is $C^{q}$--close to $g.$

\begin{lemma}
\label{local constr main lemma} Given $K,\eta >0$, $P\in $ $\mathcal{P}_{0},$
and sufficiently small $\varepsilon _{0},\rho >0,$ there is a $\xi >0$ so
that if 
\begin{equation*}
\left\vert g-\hat{g}\right\vert _{C^{q}}<\xi ,
\end{equation*}%
then there is a $C^{\infty }$--family of metrics $\left\{ g_{s}\right\}
_{s\in \left[ 0,s_{0}\right] }$ so that the following hold.

\begin{enumerate}
\item For all $s,$ $g_{s}=\hat{g}$ on $M\setminus B\left( \pi \left(
P\right) ,\rho +\eta \right) ,$ and $g_{0}=\hat{g}.$

\item Let $\sigma \left( P\right) $ be the section of $\mathcal{G}_{l}\left(
B\left( \pi \left( P\right) ,\rho \right) \right) $ determined by $P$ via
normal coordinates at $\pi \left( P\right) $ with respect to $g.$ For all%
\begin{equation}
\check{P}\in \pi ^{-1}\left( B\left( \pi \left( P\right) ,\rho \right)
\right) \cap B\left( \mathfrak{\sigma }\left( P\right) ,\rho \right) ,
\label{effected set}
\end{equation}%
there is a $v\in \check{P}$ so that%
\begin{equation*}
\left\vert \mathcal{I}_{R_{v}^{s}}\left( \check{P}\right) -\mathcal{I}%
_{R_{v}^{\hat{g}}}\left( \check{P}\right) \right\vert >Ks.
\end{equation*}%
Here $R^{s}$ and $R^{\hat{g}}$ are the curvature tensors of $g_{s}$ and $%
\hat{g},$ respectively.

\item For all $\check{P}\in \mathcal{G}_{l}\left( M\right) $ and all $v\in 
\check{P},$ 
\begin{equation*}
\left\vert \mathcal{I}_{R_{v}^{s}}\left( \check{P}\right) -\mathcal{I}%
_{R_{v}^{\hat{g}}}\left( \check{P}\right) \right\vert \leq 2Ks.
\end{equation*}

\item For all $\check{P}\in \mathcal{G}_{l}\left( M\right) \setminus \left\{
\pi ^{-1}\left( B\left( \pi \left( P\right) ,\rho +\eta \right) \right) \cap
B\left( \mathfrak{\sigma }\left( P\right) ,\rho +\eta \right) \right\} $ and 
$w\in \check{P},$ 
\begin{equation*}
\left\vert \mathcal{I}_{R_{w}^{s}}\left( \check{P}\right) -\mathcal{I}%
_{R_{w}^{\hat{g}}}\left( \check{P}\right) \right\vert \leq \varepsilon _{0}s.
\end{equation*}
\end{enumerate}
\end{lemma}

We will not need Parts 3 and 4 to prove Theorem \ref{main thm}, but have
included them since they are obtained relatively easily and seem to be of
independent interest. {\LARGE \ }

The proof of Lemma \ref{local constr main lemma} occupies the rest of this
section and starts with some preliminary results.

\begin{lemma}
\label{constr of f on M lemma}Given $K,\varepsilon ,\eta >0$, and $P\in $ $%
\mathcal{P}_{0},$ there are coordinate neighborhoods $N$ and $G$ of $\pi
\left( P\right) $ and $C^{\infty }$ functions $f_{1},f_{2}:M\longrightarrow 
\mathbb{R}$ with the following properties.

\begin{enumerate}
\item 
\begin{equation*}
\mathrm{dist}\left( N,M\setminus G\right) <\eta .
\end{equation*}

\item On $N,$ the second partial derivatives in the first coordinate
direction satisfy 
\begin{equation*}
\max \left\{ \left\vert \partial _{1}\partial _{1}f_{1}\right\vert
,\left\vert \partial _{1}\partial _{1}f_{2}\right\vert \right\} >2K.
\end{equation*}

\item In general, 
\begin{equation*}
\max \left\{ \left\vert \partial _{1}\partial _{1}f_{1}\right\vert
,\left\vert \partial _{1}\partial _{1}f_{2}\right\vert \right\} \leq 4K.
\end{equation*}

\item For $j\in \left\{ 1,2,\ldots ,n\right\} ,$ $k\in \left\{ 2,\ldots
,n\right\} ,$ and $i\in \left\{ 1,2\right\} ,$ 
\begin{equation*}
\left\vert \partial _{j}\partial _{k}f_{i}\right\vert <\varepsilon
\end{equation*}%
and%
\begin{equation*}
\left\vert f_{i}\right\vert _{C^{1}}<\varepsilon \text{.}
\end{equation*}

\item On $M\setminus G,$ $f_{1}=f_{2}=0.$
\end{enumerate}
\end{lemma}

This follows by composing the coordinate chart of $G$ with the functions on
Euclidean space given by the next lemma.

\begin{lemma}
\label{constrution of f lemma}Let $\pi _{1}:\mathbb{R}^{n}\longrightarrow 
\mathbb{R}$ be orthogonal projection onto the first factor. Let $\mathcal{C}$
be a compact subset of $\mathbb{R}^{n}$ with $\pi _{1}\left( \mathcal{C}%
\right) =\left[ a,b\right] ,$ for $a,b\in \mathbb{R}.$ Given $K,\varepsilon
>0$ and a compact set $\mathcal{\tilde{C}}$ with $\mathcal{C}\subset \mathrm{%
int}\left( \mathcal{\tilde{C}}\right) $, there are $C^{\infty }$ functions $%
f_{1},f_{2}:\mathbb{R}^{n}\longrightarrow \mathbb{R}$ with the following
properties.

\begin{enumerate}
\item On $\mathcal{C},$%
\begin{equation*}
\max \left\{ \left\vert \partial _{1}\partial _{1}f_{1}\right\vert
,\left\vert \partial _{1}\partial _{1}f_{2}\right\vert \right\} >2K.
\end{equation*}

\item In general, 
\begin{equation*}
\max \left\{ \left\vert \partial _{1}\partial _{1}f_{1}\right\vert
,\left\vert \partial _{1}\partial _{1}f_{2}\right\vert \right\} \leq 4K.
\end{equation*}

\item For $j\in \left\{ 1,2,\ldots ,n\right\} ,$ $k\in \left\{ 2,\ldots
,n\right\} ,$ and $i\in \left\{ 1,2\right\} ,$ 
\begin{equation*}
\left\vert \partial _{j}\partial _{k}f_{i}\right\vert <\varepsilon
\end{equation*}%
and%
\begin{equation*}
\left\vert f_{i}\right\vert _{C^{1}}<\varepsilon \text{.}
\end{equation*}

\item On $\mathbb{R}^{n}\setminus \mathcal{\tilde{C}},$ $f_{1}=f_{2}=0.$
\end{enumerate}
\end{lemma}

To prove Lemma \ref{constrution of f lemma} we will use the following single
variable calculus result.

\begin{lemma}
\label{jim lemma}Given any $K>1$ and $\varepsilon >0,$ there are $C^{\infty
} $ functions $h_{1},h_{2}:\mathbb{R}\longrightarrow \mathbb{R}$ so that for
all $t$ 
\begin{eqnarray}
\frac{201}{100}K &<&\max_{i\in \left\{ 1,2\right\} }\left\{ \left\vert
h_{i}^{\prime \prime }\left( t\right) \right\vert \right\} <\frac{399}{100}K%
\text{ and\label{2nd derrrv}} \\
\left\vert h_{i}\right\vert _{C^{1}} &<&\varepsilon .  \label{lots
wiggle}
\end{eqnarray}
\end{lemma}

\begin{proof}
Let $\eta =\frac{\varepsilon }{4\cdot K},$ and for $\delta >0$ set 
\begin{equation*}
h_{\delta ,1}\left( t\right) =\left( 4-\delta \right) K\eta ^{2}\sin \left( 
\frac{t}{\eta }\right) \text{ and }h_{\delta ,2}\left( t\right) =\left(
4-\delta \right) K\eta ^{2}\cos \left( \frac{t}{\eta }\right) .
\end{equation*}%
Then 
\begin{eqnarray*}
h_{\delta ,1}^{\prime }\left( t\right)  &=&\left( 4-\delta \right) K\eta
\cos \left( \frac{t}{\eta }\right) =\frac{\left( 4-\delta \right) }{4}%
\varepsilon \cos \left( \frac{t}{\eta }\right) \text{ and } \\
h_{\delta ,2}^{\prime }\left( t\right)  &=&-\left( 4-\delta \right) K\eta
\sin \left( \frac{t}{\eta }\right) =-\frac{\left( 4-\delta \right) }{4}%
\varepsilon \sin \left( \frac{t}{\eta }\right) ,
\end{eqnarray*}%
so for all $\delta \in \left( 0,4\right) ,$ 
\begin{equation*}
\left\vert h_{\delta ,1}\right\vert _{C^{1}}<\varepsilon \text{ and }%
\left\vert h_{\delta ,2}\right\vert _{C^{1}}<\varepsilon .
\end{equation*}

Also 
\begin{equation*}
h_{\delta ,1}^{\prime \prime }\left( t\right) =-\left( 4-\delta \right)
K\sin \left( \frac{t}{\eta }\right) \text{ and }h_{\delta ,2}^{\prime \prime
}\left( t\right) =-\left( 4-\delta \right) K\cos \left( \frac{t}{\eta }%
\right) .
\end{equation*}%
Since for all $t,$%
\begin{equation*}
\frac{1}{\sqrt{2}}\leq \max \left\{ \left\vert \sin \left( t\right)
\right\vert ,\left\vert \cos \left( t\right) \right\vert \right\} \leq 1,
\end{equation*}
\begin{equation*}
\frac{\left( 4-\delta \right) }{\sqrt{2}}K\leq \max_{i\in \left\{
1,2\right\} }\left\{ \left\vert h_{\delta ,i}^{\prime \prime }\left(
t\right) \right\vert \right\} \leq \left( 4-\delta \right) K,
\end{equation*}%
and (\ref{2nd derrrv}) holds with $h_{1}=h_{\delta ,1}$ and $h_{2}=h_{\delta
,2},$ provided 
\begin{equation*}
\sqrt{2}\frac{201}{100}<\left( 4-\delta \right) <\frac{399}{100}.
\end{equation*}
\end{proof}

\begin{proof}[Proof of Lemma \protect\ref{constrution of f lemma}]
Let $\chi :\mathbb{R}^{n}\longrightarrow \left[ 0,1\right] $ be $C^{\infty }$
and satisfy 
\begin{eqnarray}
\chi |_{\mathcal{C}} &\equiv &1,  \notag \\
\chi |_{\mathbb{R}^{n}\setminus \mathcal{\tilde{C}}} &\equiv &0.
\label{dfn
chi}
\end{eqnarray}%
Let $M>1$ satisfy 
\begin{equation}
\left\vert \chi \right\vert _{C^{2}}\leq M.  \label{C2 norm of chi}
\end{equation}

For $\tilde{\varepsilon}\in \left( 0,\varepsilon \right) ,$ we use Lemma \ref%
{jim lemma} to choose $C^{\infty }$--functions $h_{1},h_{2}:\mathbb{R}%
\longrightarrow \mathbb{R}$ that satisfy 
\begin{equation}
\left\vert h_{i}\right\vert _{C^{1}}<\frac{\tilde{\varepsilon}}{2M}
\label{C1 norm}
\end{equation}%
and 
\begin{equation}
\frac{201}{100}K<\max_{i\in \left\{ 1,2\right\} }\left\{ \left\vert
h_{1}^{\prime \prime }\left( t\right) \right\vert ,\left\vert h_{2}^{\prime
\prime }\left( t\right) \right\vert \right\} <\frac{399}{100}K.
\label{h1 2nd deriv}
\end{equation}

For $i=1,2,$ we set 
\begin{equation}
f_{i}\left( p\right) =\chi \left( p\right) \cdot \left( h_{i}\circ \pi
_{1}\right) \left( p\right) .  \label{dfn of f}
\end{equation}%
Then 
\begin{equation*}
\partial _{k}f_{i}\left( p\right) =\partial _{k}\chi \left( p\right) \cdot
\left( h_{i}\circ \pi _{1}\right) \left( p\right) +\chi \left( p\right)
\cdot \partial _{k}\left( h_{i}\circ \pi _{1}\right) \left( p\right)
\end{equation*}%
and 
\begin{eqnarray*}
\partial _{j}\partial _{k}f_{i}\left( p\right) &=&\partial _{j}\partial
_{k}\chi \left( p\right) \cdot \left( h_{i}\circ \pi _{1}\right) \left(
p\right) +\partial _{j}\chi \left( p\right) \cdot \partial _{k}\left(
h_{i}\circ \pi _{1}\right) \left( p\right) \\
&&+\partial _{k}\chi \left( p\right) \cdot \partial _{j}\left( h_{i}\circ
\pi _{1}\right) \left( p\right) +\chi \left( p\right) \cdot \partial
_{j}\partial _{k}\left( h_{i}\circ \pi _{1}\right) \left( p\right) .
\end{eqnarray*}%
Combining this with (\ref{dfn chi}), (\ref{C2 norm of chi}), (\ref{C1 norm}%
), and (\ref{h1 2nd deriv}) we see that Properties 1, 2, and 3 hold,
provided $\tilde{\varepsilon}$ is sufficiently small. Property 4 follows
from (\ref{dfn chi}) and (\ref{dfn of f}).
\end{proof}

Let $P\in \mathcal{P}_{0}$ be as in Lemma \ref{local constr main lemma} and
have foot point $p.$ If $l=2,$ let $\left\{ v,T,n_{3},n_{4}\right\} $ be an
ordered orthonormal quadruplet at $p$ with%
\begin{equation}
\left\{ v,T\right\} \in P\text{ and }\left\{ n_{3},n_{4}\right\} \text{
normal to }P.  \label{surf choice}
\end{equation}%
If $l=n-1,$ let $\left\{ v,n,T_{3},T_{4}\right\} $ be an ordered $\hat{g}$%
--orthonormal quadruplet at $p$ with 
\begin{equation}
\left\{ v,T_{3},T_{4}\right\} \in P\text{ and }n\text{ normal to }P.
\label{hyp surf choice}
\end{equation}%
If $l\in \left\{ 3,\ldots ,n-2\right\} ,$ let $\left\{ E_{i}\right\}
_{i=1}^{4}$ be an ordered $\hat{g}$--orthonormal quadruplet at $p$ that
satisfies either $\left( \ref{surf choice}\right) $ or $\left( \ref{hyp surf
choice}\right) $. In either case, extend the ordered quadruplet to a
coordinate frame $\left\{ E_{i}\right\} _{i=1}^{n}$.

Choose $g_{s}$ so that with respect to the ordered frame $\left\{
E_{i}\right\} _{i=1}^{n},$ the matrix of $g_{s}-\hat{g}$ is $0$ except for
the upper $\left( 4\times 4\right) $--block which is 
\begin{equation}
\left\{ g_{s}-\hat{g}\right\} _{l,m}=\left( 
\begin{array}{llll}
0 & 0 & 0 & 0 \\ 
0 & 0 & sf_{1} & sf_{2} \\ 
0 & sf_{1} & 0 & 0 \\ 
0 & sf_{2} & 0 & 0%
\end{array}%
\right) ,  \label{dfn ofg tilde}
\end{equation}%
where to construct $f_{1}$ and $f_{2}$ we apply Lemma \ref{constr of f on M
lemma} with $N=B\left( \pi \left( P\right) ,\rho \right) .$

To simplify notation, we write $\tilde{g}$ for $g_{s}$ and use $\widetilde{}$
for objects associated to $\tilde{g}.$ Recall (see, e.g., \cite{Pet}) that
with respect to $\left\{ E_{i}\right\} _{i=1}^{n},$ the Christoffel symbols
are 
\begin{equation*}
\tilde{\Gamma}_{ij,k}\equiv \tilde{g}\left( \tilde{\nabla}%
_{E_{i}}E_{j},E_{k}\right) 
\end{equation*}%
and 
\begin{equation}
\tilde{R}_{ijkl}=\partial _{i}\tilde{\Gamma}_{jk,l}-\partial _{j}\tilde{%
\Gamma}_{ik,l}+\tilde{g}^{\sigma \tau }\left( \tilde{\Gamma}_{ik,\sigma }%
\tilde{\Gamma}_{jl,\tau }-\tilde{\Gamma}_{jk,\sigma }\tilde{\Gamma}_{il,\tau
}\right) ,  \label{curv via christ eqn}
\end{equation}%
where $\tilde{g}^{\sigma \tau }$ are the coefficients of the inverse $\left(
\left\{ \tilde{g}\right\} _{\sigma \tau }\right) ^{-1}$ of $\left\{ \tilde{g}%
\right\} _{\sigma \tau },$ and the Einstein summation convention is being
used. Combining Lemma \ref{constr of f on M lemma} with the definition of $%
\tilde{g}$ gives us

\begin{proposition}
\label{g inv prop}The coefficients $\hat{g}^{l,m}$ and $\tilde{g}^{l,m}$ of
the inverses of $\left\{ \hat{g}\right\} _{l,m}$ and $\left\{ \tilde{g}%
\right\} _{l,m}$ satisfy%
\begin{equation*}
\left\vert \hat{g}^{l,m}-\tilde{g}^{l,m}\right\vert <O\left( \varepsilon
s\right) .
\end{equation*}
\end{proposition}

Using Equation (\ref{dfn ofg tilde}) and Lemma \ref{constr of f on M lemma},
we will show

\begin{proposition}
\label{Gamma control prop}Writing $\left( \tilde{\Gamma}-\hat{\Gamma}\right)
_{jk,l}$ for $\tilde{\Gamma}_{ij,k}-\hat{\Gamma}_{ij,k}$ we have 
\begin{equation}
\left\vert \left( \tilde{\Gamma}-\hat{\Gamma}\right) _{jk,l}\right\vert
<O\left( \varepsilon s\right) .  \label{Gamma diff inequal}
\end{equation}

Let $i,j,k,l$ be arbitrary elements of $\left\{ 1,2,\ldots ,n\right\} .$
Then all expressions 
\begin{equation*}
\partial _{i}\left( \tilde{\Gamma}-\hat{\Gamma}\right) _{jk,l}
\end{equation*}%
are $\leq O\left( \varepsilon s\right) $ except for%
\begin{equation}
\begin{array}{ccccc}
\partial _{v}\left( \tilde{\Gamma}-\hat{\Gamma}\right) _{2v,3} & = & 
\partial _{v}\left( \tilde{\Gamma}-\hat{\Gamma}\right) _{v3,2} & = & 
-\partial _{v}\left( \tilde{\Gamma}-\hat{\Gamma}\right) _{23,v} \\ 
\shortparallel &  & \shortparallel &  & \shortparallel \\ 
\partial _{v}\left( \tilde{\Gamma}-\hat{\Gamma}\right) _{v2,3} & = & 
\partial _{v}\left( \tilde{\Gamma}-\hat{\Gamma}\right) _{3v,2} & = & 
-\partial _{v}\left( \tilde{\Gamma}-\hat{\Gamma}\right) _{32,v}%
\end{array}
\label{big curv}
\end{equation}%
and%
\begin{equation}
\begin{array}{ccccc}
\partial _{v}\left( \tilde{\Gamma}-\hat{\Gamma}\right) _{2v,4} & = & 
\partial _{v}\left( \tilde{\Gamma}-\hat{\Gamma}\right) _{v4,2} & = & 
-\partial _{v}\left( \tilde{\Gamma}-\hat{\Gamma}\right) _{24,v} \\ 
\shortparallel &  & \shortparallel &  & \shortparallel \\ 
\partial _{v}\left( \tilde{\Gamma}-\hat{\Gamma}\right) _{v2,4} & = & 
\partial _{v}\left( \tilde{\Gamma}-\hat{\Gamma}\right) _{4v,2} & = & 
-\partial _{v}\left( \tilde{\Gamma}-\hat{\Gamma}\right) _{42,v},%
\end{array}
\label{other big curv}
\end{equation}%
where we write $v$ for the first element of our frame to emphasize its
special role. The expressions in (\ref{big curv}) and (\ref{other big curv})
are $\leq 2\tilde{K}s$ everywhere, and on $N,$%
\begin{equation*}
\max \left\{ \left( \ref{big curv}\right) ,\left( \ref{other big curv}%
\right) \right\} \geq Ks.
\end{equation*}
\end{proposition}

\begin{proof}
Inequality (\ref{Gamma diff inequal}) follows from the fact that $\left\vert
sf_{i}\right\vert _{C^{1}}<\varepsilon s.$

To prove the remainder note 
\begin{eqnarray*}
\partial _{i}\tilde{\Gamma}_{jk,l} &=&\partial _{E_{i}}\tilde{g}\left( 
\tilde{\nabla}_{E_{j}}E_{k},E_{l}\right) \\
&=&\frac{1}{2}\partial _{E_{i}}\left[ \partial _{E_{k}}\tilde{g}\left(
E_{j},E_{l}\right) +\partial _{E_{j}}\tilde{g}\left( E_{l},E_{k}\right)
-\partial _{E_{l}}\tilde{g}\left( E_{k},E_{j}\right) \right] \\
&=&\frac{1}{2}\partial _{E_{i}}\left[ \partial _{E_{k}}\left( \tilde{g}-\hat{%
g}\right) \left( E_{j},E_{l}\right) +\partial _{E_{j}}\left( \tilde{g}-\hat{g%
}\right) \left( E_{l},E_{k}\right) -\partial _{E_{l}}\left( \tilde{g}-\hat{g}%
\right) \left( E_{k},E_{j}\right) \right] \\
&&+\frac{1}{2}\partial _{E_{i}}\left[ \partial _{E_{k}}\hat{g}\left(
E_{j},E_{l}\right) +\partial _{E_{j}}\hat{g}\left( E_{l},E_{k}\right)
-\partial _{_{E_{l}}}\hat{g}\left( E_{k},E_{j}\right) \right] .
\end{eqnarray*}%
Combining this with Lemma \ref{constr of f on M lemma} and the definition of 
$\tilde{g}$ gives us 
\begin{eqnarray*}
\partial _{i}\tilde{\Gamma}_{jk,l} &=&\partial _{E_{i}}\hat{g}\left( \nabla
_{E_{j}}E_{k},E_{l}\right) +O\left( \varepsilon s\right) \\
&=&\partial _{i}\hat{\Gamma}_{jk,l}+O\left( \varepsilon s\right) ,
\end{eqnarray*}%
unless the indices correspond to the situation in (\ref{big curv}) or (\ref%
{other big curv}). In the former case,%
\begin{eqnarray*}
\partial _{v}\left( \tilde{\Gamma}-\hat{\Gamma}\right) _{2v,3} &=&\frac{1}{2}%
\partial _{E_{v}}\left[ \partial _{E_{v}}\left( \tilde{g}-\hat{g}\right)
\left( E_{2},E_{3}\right) +\partial _{E_{2}}\left( \tilde{g}-\hat{g}\right)
\left( E_{3},E_{v}\right) -\partial _{E_{3}}\left( \tilde{g}-\hat{g}\right)
\left( E_{v},E_{2}\right) \right] \\
&=&\frac{s}{2}\partial _{E_{v}}\partial _{E_{v}}\left( f_{1}\right) .
\end{eqnarray*}

In the case of (\ref{other big curv}) we have 
\begin{eqnarray*}
\partial _{v}\left( \tilde{\Gamma}-\hat{\Gamma}\right) _{2v,4} &=&\frac{1}{2}%
\partial _{E_{v}}\left[ \partial _{E_{v}}\left( \tilde{g}-\hat{g}\right)
\left( E_{2},E_{4}\right) +\partial _{E_{2}}\left( \tilde{g}-\hat{g}\right)
\left( E_{4},E_{v}\right) -\partial _{E_{4}}\left( \tilde{g}-\hat{g}\right)
\left( E_{v},E_{2}\right) \right] \\
&=&\frac{s}{2}\partial _{E_{v}}\partial _{E_{v}}\left( f_{2}\right) .
\end{eqnarray*}

The result follows by combining the previous three displays with Lemma \ref%
{constr of f on M lemma}.
\end{proof}

\begin{proof}[ Proof of Lemma \protect\ref{local constr main lemma}]
Combining Propositions \ref{g inv prop} and \ref{Gamma control prop} with
Equation (\ref{curv via christ eqn}) we see that 
\begin{equation*}
\left\vert \left( \tilde{R}-\hat{R}\right) _{ijkl}\right\vert \leq O\left(
\varepsilon s\right) ,
\end{equation*}%
except if the quadruple corresponds, up to a symmetry of the curvature
tensor, to either $\left( \tilde{R}-\hat{R}\right) _{2vv3}$ or $\left( 
\tilde{R}-\hat{R}\right) _{2vv4},$ in which case we have 
\begin{eqnarray*}
\left( \tilde{R}-\hat{R}\right) _{2vv3} &=&\partial _{v}\partial _{v}\left(
f_{1}\right) +O\left( \varepsilon s\right) \text{ and} \\
\left( \tilde{R}-\hat{R}\right) _{2vv4} &=&\partial _{v}\partial _{v}\left(
f_{2}\right) +O\left( \varepsilon s\right) .
\end{eqnarray*}

Lemma \ref{local constr main lemma} follows from the previous two equations
and our choices of $f_{1}$ and $f_{2},$ provided $\varepsilon $ is
sufficiently small.
\end{proof}

\section{The Global Construction\label{global constr}}

In this section, we prove the $l^{th}$--Partially Geodesic Assertion and
hence Theorems \ref{main thm} and \ref{partiall geod thm}. The proof is by%
\textbf{\ }reverse induction, starting with the case when $l=n-1.$ The
strategy is to apply Lemma \ref{local constr main lemma} successively to the
elements of an open cover of $\mathcal{P}_{0}.$ When $l=n-1,$ this is all
that is needed. Otherwise, as in the proof of Lemma \ref{uber stru lemma},
we note that for each $k\in \left\{ l+1,\ldots ,n-1\right\} $, $\mathcal{G}%
_{k}\left( M\right) $ is compact. By our induction hypothesis, there is a $%
\delta >0$ so that for all $P\in \mathcal{G}_{k}\left( M\right) $ there is a
unit $v\in P$ with 
\begin{equation*}
\mathcal{I}_{R_{v}^{0}}\left( P\right) >\delta .
\end{equation*}%
Thus all sufficiently small deformations of $g$ have no partially geodesic $%
k $--planes for all $k\in \left\{ l+1,\ldots ,n-1\right\} .$ In particular,
the $l^{th}$--Partially Geodesic Assertion follows from the \vspace{2em}

\noindent \textbf{Modified} $\mathbf{l}^{th}$--\textbf{Partially Geodesic
Assertion. }\emph{\ Given }$l\in \left\{ 2,3,\ldots ,n-1\right\} ,$\emph{\ a
finite }$q\geq 2,$ \emph{and }$\xi >0$, \emph{there is a Riemannian metric }$%
\tilde{g}$ \emph{on} $M$ \emph{with} 
\begin{equation*}
\left\vert \tilde{g}-g\right\vert _{C^{q}}<\xi
\end{equation*}%
\emph{that has no partially geodesic }$l$\emph{--planes.}

\begin{proof}
Given $K>0,$ we combine Lemma \ref{local constr main lemma} with the
compactness of $\mathcal{P}_{0}$ to see that there is a finite open cover $%
\left\{ \pi ^{-1}\left( B\left( \pi \left( P_{i}\right) ,\rho _{i}\right)
\right) \cap B\left( \mathfrak{\sigma }\left( P_{i}\right) ,\rho _{i}\right)
\right\} _{i=1}^{G}$ of $\mathcal{P}_{0}$ whose elements are as in (\ref%
{effected set}). In particular, for each $i\in \left\{ 2,3,\ldots ,G\right\}
,$ there is a $\xi _{i}>0$ so that if 
\begin{equation*}
\left\vert g-\hat{g}\right\vert _{C^{q}}<\xi _{i},
\end{equation*}%
then the conclusion of Lemma \ref{local constr main lemma} holds on $\pi
^{-1}\left( B\left( \pi \left( P_{i}\right) ,\rho _{i}\right) \right) \cap
B\left( \mathfrak{\sigma }\left( P_{i}\right) ,\rho _{i}\right) .$ Set 
\begin{equation*}
\xi =\min\nolimits_{i}\left\{ \xi _{i}\right\} .
\end{equation*}%
Since $\mathcal{G}_{l}\left( M\right) \setminus \left\{
\bigcup\limits_{i=1}^{G}\pi ^{-1}\left( B\left( \pi \left( P_{i}\right)
,\rho _{i}\right) \right) \cap B\left( \mathfrak{\sigma }\left( P_{i}\right)
,\rho _{i}\right) \right\} $ is compact, there is a $\delta >0$ so that for
all 
\begin{equation*}
P\in \mathcal{G}_{l}\left( M\right) \setminus \left\{
\bigcup\limits_{i=1}^{G}\pi ^{-1}\left( B\left( \pi \left( P_{i}\right)
,\rho _{i}\right) \right) \cap B\left( \mathfrak{\sigma }\left( P_{i}\right)
,\rho _{i}\right) \right\} ,
\end{equation*}%
there is a $v\in P$ so that%
\begin{equation}
\left\vert \mathcal{I}_{R_{v}^{g}}\left( P\right) \right\vert >\delta .
\label{lemma D all over again}
\end{equation}

We will successively apply Lemma \ref{local constr main lemma} to the $\pi
^{-1}\left( B\left( \pi \left( P_{i}\right) ,\rho _{i}\right) \right) \cap
B\left( \mathfrak{\sigma }\left( P_{i}\right) ,\rho _{i}\right) $ and get a
sequence of metrics $g_{1},$ $g_{2},$ $\ldots ,g_{G}.$ To obtain $g_{1},$ we
apply Lemma \ref{local constr main lemma} with $g=\hat{g}$, $P=P_{1},$ and $%
\rho =\rho _{1}.$ This yields a deformation $g_{s}$ of $g.$ Let $\mathcal{P}%
_{s}$ be the set of partially geodesic $l$--planes for $g_{s}.$ It follows
from Part 2 of Lemma \ref{local constr main lemma} that for all sufficiently
small $s,$ 
\begin{equation}
\text{ }\mathcal{P}_{s}\dbigcap \pi ^{-1}\left( B\left( \pi \left(
P_{1}\right) ,\rho _{1}\right) \right) \cap B\left( \mathfrak{\sigma }\left(
P_{1}\right) ,\rho _{1}\right) =\emptyset .  \label{Lemma 21 first app}
\end{equation}%
By combining (\ref{lemma D all over again}) and (\ref{Lemma 21 first app}),
we see that for sufficiently small $s,$ 
\begin{equation*}
\mathcal{P}_{s}\subset \bigcup\limits_{i=2}^{G}\pi ^{-1}\left( B\left( \pi
\left( P_{i}\right) ,\rho _{i}\right) \right) \cap B\left( \mathfrak{\sigma }%
\left( P_{i}\right) ,\rho _{i}\right) .
\end{equation*}

Moreover, by further restricting $s,$ we can ensure that $g_{s}$ is close
enough to $g$ in the $C^{q}$--topology so that%
\begin{equation*}
\left\vert g_{s}-g\right\vert _{C^{q}}<\xi .\text{ }
\end{equation*}

We let $g_{1}=g_{s}$ for some $s$ as above. Assume, by induction, that for
some $k\in \left\{ 1,\ldots ,G-1\right\} ,$ we have constructed a metric $%
g_{k}$ so that the following hold:

$\left( \mathbf{Hypothesis}\text{ }\mathbf{1}_{k}\right) $ If $\mathcal{P}%
_{g_{k}}$ is the set of $l$--dimensional partially geodesic subspaces for $%
g_{k},$ then 
\begin{equation*}
\mathcal{P}_{g_{k}}\subset \bigcup\limits_{i=k+1}^{G}\pi ^{-1}\left( B\left(
\pi \left( P_{i}\right) ,\rho _{i}\right) \right) \cap B\left( \mathfrak{%
\sigma }\left( P_{i}\right) ,\rho _{i}\right) .
\end{equation*}

$\left( \mathbf{Hypothesis}\text{ }\mathbf{2}_{k}\right) $%
\begin{equation*}
\left\vert g_{k}-g\right\vert _{C^{q}}<\xi .
\end{equation*}

It follows from Hypothesis $1_{k}$ that there is $\delta >0$ so that for all 
\begin{equation*}
P\in \mathcal{G}_{l}\left( M\right) \setminus \left\{
\bigcup\limits_{i=k+1}^{G}\pi ^{-1}\left( B\left( \pi \left( P_{i}\right)
,\rho _{i}\right) \right) \cap B\left( \mathfrak{\sigma }\left( P_{i}\right)
,\rho _{i}\right) \right\} ,
\end{equation*}%
there is a $v\in P$ so that 
\begin{equation}
\left\vert \mathcal{I}_{R_{v}^{g_{k}}}\left( P\right) \right\vert >\delta .
\label{old dist}
\end{equation}%
Since $\left\vert g_{k}-g\right\vert _{C^{q}}<\xi ,$ we can apply Lemma \ref%
{local constr main lemma} with $\hat{g}=g_{k}$ , $P=P_{k+1},$ and $\rho
=\rho _{k+1}.$ This yields a deformation $g_{s}$ of $g_{k}$ so that for all
sufficiently small $s>0,$

\begin{equation*}
\left\vert g_{s}-g\right\vert _{C^{q}}<\xi .
\end{equation*}%
In other words, Hypothesis $2_{k+1}$ holds.

To establish Hypothesis $1_{k+1},$ combine Part 2 of Lemma \ref{local constr
main lemma} with (\ref{old dist}) to see that $g_{s}$ has no partially
geodesic $l$--dimensional subspaces in%
\begin{equation*}
\mathcal{G}_{l}\left( M\right) \setminus \left\{
\bigcup\limits_{i=k+2}^{G}\pi ^{-1}\left( B\left( \pi \left( P_{i}\right)
,\rho _{i}\right) \right) \cap B\left( \mathfrak{\sigma }\left( P_{i}\right)
,\rho _{i}\right) \right\} ,
\end{equation*}%
provided $s$ is positive and sufficiently small. So Hypothesis $1_{k+1}$
also holds.
\end{proof}

\end{document}